\documentclass{amsart}
\usepackage{amsfonts}
\usepackage{amsmath}
\usepackage{amssymb}
\usepackage{amsthm}

\numberwithin{equation}{section}

\newtheorem{theorem}{Theorem}[section]
\newtheorem{proposition}[theorem]{Proposition}
\newtheorem{lemma}[theorem]{Lemma}
\newtheorem{corollary}[theorem]{Corollary}
\newtheorem{conjecture}[theorem]{Conjecture}

\theoremstyle{definition}
\newtheorem{definition}[theorem]{Definition}

\begin{document}


\title[Fully weighted zero-sum subsequences]{On the number of fully weighted zero-sum subsequences}

\author{Ab\' ilio  Lemos, Allan O. Moura, Anderson T. Silva and B. K. Moriya}

\address{Departamento de Matem\'{a}tica, Universidade Federal de Vi\c cosa, Vi\c cosa-MG, Brazil}

\email{abiliolemos@ufv.br\\allan.moura@ufv.br\\anderson.tiago@ufv.br}
\email{bhavinkumar@ufv.br}



\begin{abstract}
Let $G$ be a finite additive abelian group with exponent $n$ and $S=g_{1}\cdots g_{t}$
be a sequence of elements in $G$. For any element $g$ of $G$ and $A\subseteq\{1,2,\ldots,n-1\}$,
let $N_{A,g}(S)$ denote the number of subsequences $T=\prod_{i\in I}g_{i}$
of $S$ such that $\sum_{i\in I}a_{i}g_{i}=g$ , where $I\subseteq\left\{ 1,\ldots,t\right\} $
and $a_{i}\in A$. In this paper, we prove that $N_{A,0}(S)\geq2^{|S|-D_{A}(G)+1}$,
when $A=\left\{ 1,\ldots,n-1\right\} $, where $D_{A}(G)$ is the
smallest positive integer $l$, such that every sequence $S$ over
$G$ of length at least $l$ has nonempty subsequence $T=\prod_{i\in I}g_{i}$
such that $\sum_{i\in I}a_{i}g_{i}=0$, $I\subseteq\left\{ 1,\ldots,t\right\} $
and $a_{i}\in A$. Moreover, we classify the sequences such that $N_{A,0}(S)=2^{|S|-D_{A}(G)+1}$, where the exponent of $G$ is an odd number.
\end{abstract}

\maketitle

\section{Introduction}	

\hspace{0.6cm}Let $G$ be a finite additive abelian group with exponent
$n$ and $S$ be a sequence over $G$. The enumeration of subsequences
with certain prescribed properties is a classical topic in Combinatorial
Number Theory going back to Erd\"{o}s, Ginzburg and Ziv (see \cite{EGZ,Ger1,Ger2})
who proved that $2n-1$ is the smallest integer, such that every sequence
$S$ over a cyclic group $C_{n}$ has a subsequence of length $n$
with zero-sum. This raises the problem of determining the smallest
positive integer $l$, such that every sequence $S=g_{1}\cdots g_{l}$
has a nonempty zero-sum subsequence. Such an integer $l$ is called
the {\it Davenport constant of $G$} (see \cite{Dav, OlsonI}), denoted by $D(G)$,
which is still unknown for wide class of groups. In an analogous manner, for a nonempty subset $A\subseteq\mathbb{Z}\backslash\left\{ kn:k\in\mathbb{Z}\right\} $, Adhikari \textit{et. al.} (see \cite{Adh1}) defined an $A$-weighted Davenport constant, denoted by $D_A(G)$, to be a smallest $t\in \mathbb{N}$ such that every sequence $S$ over $G$ of length $t$ has nonempty $A$-weighted zero-sum subsequence.

For any $g$ of $G$, let $N_{A,g}(S)$ (when $A=\{1\}$ we write
$N_{g}(S)$) denote {\it the number of weighted subsequences} $T=\prod_{i\in I}g_{i}$
of $S=g_{1}\cdots g_{l}$ such that $\sum_{i\in I}a_{i}g_{i}=g$,
where $I\subseteq\left\{ 1,\ldots,l\right\} $ is a nonempty subset
and $a_{i}\in A$. In 1969, \cite{OlsonII} proved
that $N_{0}(S)\geq2^{|S|-D(G)+1}$ for every sequence $S$ over $G$
of length $|S|\geq D(G)$. Subsequently, several authors, including
\cite{Bal,Bia,Cao,Fur,Gao2,Gao3,Gao5,Gao6,Gry1,Gry2,Gui,kis} obtained a huge variety of results on the number of subsequences with
prescribed properties. Recently, Chang {\it et al.} \cite{Chang} found the lower bound
of $N_{g}(S)$ for any arbitrary $g$.

In this paper, we determine a lower bound for $N_{A,0}(S)$, where\linebreak{}
$A=\{1,2,\ldots,n-1\}$ and $G$ is an additive finite abelian group with exponent
$n$. We also characterize the structures of the extremal sequences
which attain the lower bound for the group $G$, such that $n$ is an odd number.

\section{Notations and terminologies}

\hspace{0.6cm} In this section, we will introduce some notations and terminologies. Let $\mathbb{N}_{0}$ be {\it the set
of non-negative integers}. For integers $a,b\in\mathbb{N}_{0}$, we
define $[a,b]=\left\{ x\in\mathbb{N}_{0}:a\leq x\leq b\right\} $.

For a sequence 
\[
S=\prod_{i=1}^{m}g_{i}\in\mathcal{F}\left(G\right),
\]
where $\mathcal{F}\left(G\right)$ is {\it the free abelian monoid} with
basis $G$, a \textit{subsequence} $T=g_{i_{1}\cdots}g_{i_{k}}$ of $S$, with
$I_{T}=\{i_{1},\ldots,i_{k}\}\subseteq[1,m]$ is denoted by $T|S$;
we identify two subsequences $S_{1}$ and $S_{2}$ if $I_{S_{1}}=I_{S_{2}}$.
Given subsequences $S_{1},\ldots,S_{r}$ of $S$, we define \textit{$\gcd(S_{1},\ldots,S_{r})$} to be the sequence indexed by 
$I_{S_{1}}\cap\cdots\cap I_{S_{r}}.$ We say that two subsequences
$S_{1}$ and $S_{2}$ are {\it disjoint} if $\gcd(S_{1},S_{2})=\lambda$, where
$\lambda$ refers to the \textit{empty sequence}. If $S_{1}$ and $S_{2}$
are disjoint, then we denote by $S_{1}S_{2}$ the subsequence with
set index $I_{S_{1}}\cup I_{S_{2}}$; if $S_{1}|S_{2}$; we denote
by $S_{2}S_{1}^{-1}$ the subsequence with set index $I_{S_{2}}\setminus I_{S_{1}}$.

Moreover, we define 
\begin{enumerate}
\item $\left|S\right|=m$ {\it the length of $S$}.
\item an {\it $A$-weighted sum} is a sum of the form $\sigma^{\mathbf{a}}\left(S\right)=\sum_{i=1}^{m}a_{i}g_{i}$, with
$\mathbf{a}=a_{1}\cdots a_{m}\in\mathcal{F}(A)$, where $\mathcal{F}\left(A\right)$
is the free abelian monoid with basis $A$. When $A=[1,n-1]$, we call $S$ a {\it fully weighted sequence}.  
\item $\sum_{A}\left(S\right)=\left\{ \sum_{i\in I}a_{i}g_{i}:\emptyset\neq I\subseteq\left[1,m\right]\mbox{ and }a_{i}\in A\right\} $, a {\it set of nonempty $A$-weighted subsums of $S$}. 
\end{enumerate}
According to the above definitions, we adopt the convention that $\sigma^{\mathbf{a}}\left(\lambda\right)=0$,
for any $\mathbf{a}\in\mathcal{F}(A)$. For convenience, we define
$\sum_{A}^{\bullet}\left(S\right)=\sum_{A}\left(S\right)\cup\left\{ 0\right\} $.

The sequence $S$ is called 
\begin{enumerate}
\item an {\it $A$-weighted zero-sum free sequence} if $0\notin\sum_{A}\left(S\right)$, 
\item an {\it $A$-weighted zero-sum sequence} if $\sigma^{\mathbf{a}}\left(S\right)=0$
for some $\mathbf{a}\in\mathcal{F}(A)$. 
\end{enumerate}
For an element $g\in G$, let 
\[
N_{A,g}\left(S\right)=\left|\left\{ I\subseteq\left[1,m\right]:\sum_{i\in I}a_{i}g_{i}=g,\:a_{i}\in A\right\} \right|
\]
denote {\it the number of subsequences $T$ of $S$} with $\sigma^{\mathbf{a}}\left(T\right)=g$
for some $\mathbf{a}\in\mathcal{F}(A)$. 
\begin{definition}
Let $n$ be the exponent of $G$, $g\in G$, $A\subseteq\mathbb{Z}\backslash\left\{ kn:k\in\mathbb{Z}\right\} $ and $S\in\mathcal{F}\left(G\right)$. We say $S$ is {\it $g$-complete
sequence with weight in $A$} if $N_{A,g}\left(S\right)\geq2^{\left|S\right|-D_{A}\left(G\right)+1}$. We call $S$ an {\it extremal $g$-complete sequence with respect to $A$} if\linebreak{} $N_{A,g}\left(S\right)=2^{\left|S\right|-D_{A}\left(G\right)+1}$.
Let us denote $C_{A,g}\left(\mathcal{F}\left(G\right)\right)$ as
the {\it set of all $g$-complete sequences with respect to $A$} and $EC_{A,g}\left(\mathcal{F}\left(G\right)\right)$
as the {\it set of all extremal $g$-complete sequences with respect to $A$}. 
\end{definition}
\vspace{0cm}

\begin{definition}
Let $n$ be the exponent of $G$ and $A\subseteq\mathbb{Z}\backslash\left\{ kn:k\in\mathbb{Z}\right\} $.
We say $G$ is a {\it $0$-complete group with respect to $A$} if $\mathcal{F}\left(G\right)=C_{A,0}\left(\mathcal{F}\left(G\right)\right)$. 
\end{definition}
When $A=\left\{ 1\right\} $, Olson \cite{OlsonII} proved that all finite
abelian groups are $0$-complete with respect to $A$. Chang {\it et al.} \cite{Chang}
proved, that, when $A=\left\{ 1\right\} $, if $g\in\sum_{A}^{\bullet}\left(S\right)$,
then $S\in C_{A,g}\left(\mathcal{F}\left(G\right)\right)$ and, if
$S$ is extremal $h$-complete sequence with respect to $A$ for some
$h\in G$, then $S$ is $g$-complete sequence with respect to $A$
for all $g\in G$. Moreover, they classified the sequences in $EC_{A,0}\left(\mathcal{F}\left(G\right)\right)$
to $G$ with odd order.

Here we present an example: Take an $A$-weighted zero-sum free sequence $U$
over $G$ with $|U|=D_{A}(G)-1$. Thus, for $S=U0^{\left|S\right|-D_{A}(G)+1}$
and for all $g\in\sum_{A}^{\bullet}\left(U\right)$, we have $S\in C_{A,g}\left(\mathcal{F}\left(G\right)\right)$
and $S\in EC_{A,0}\left(\mathcal{F}\left(G\right)\right)$.

\section{Lower bound}

We write the finite abelian group $G$ as direct sum $G=H\oplus C_{n}^{r}$
, where $C_{n}^{r}$ denotes $r$ copies of the cyclic group of order $n$ denoted by $C_{n}$ and $H=C_{n_{1}}\oplus\cdots\oplus C_{n_{t}}$ with $1<n_{1}|n_{2}|\cdots|n_{t}|n=exp(G)$
and $n_{t}<n$. 

To prove our first theorem, we need some auxiliary results, which are
as follows. 

\begin{proposition}\label{prop:Proposition-2.3}{[}Proposition 2.3 \cite{marc}{]}
 Let $G$ be a finite abelian group with exponent $n$, $A\subseteq\left[1,n-1\right]$
a nonempty subset and $b\in\mathbb{N}$. Then, $$D_{bA}(G)=D_{A}(\gcd(b,n)G).$$ 
\end{proposition}
\begin{corollary}
\label{Cor.2} Let $G$ be a finite abelian group with exponent $n$
and $A\subseteq\left[1,n-1\right]$ be a nonempty subset, such that
$A\cup\left\{ 0\right\} $ is a proper subgroup of $\mathbb{Z}_{n}$.
Then, $D_{A}(G)=D_{A'}(dG)$, where $A=dA'$ and $A'\cup\left\{ 0\right\} \cong\mathbb{Z}_{n/d}$.
\end{corollary}
\begin{proof}
Since $A\cup\left\{ 0\right\} $ is a proper subgroup of $\mathbb{Z}_{n}$,
then there is $d\in\mathbb{Z}_{n}$ such that $A\cup\left\{ 0\right\} =\left\langle d\right\rangle $
and $d$ is the minimum with this property. We observe that $d=\gcd\left(d,n\right)$.
Now, we apply the Proposition \ref{prop:Proposition-2.3} with $b=d$
and obtain $D_{A}(G)=D_{A'}(dG)$, where $A=dA'$. Since $d\mathbb{Z}_{n/d}=d\left\{ 0,1,2,\dots,\dfrac{n}{d}-1\right\} \cong\left\langle d\right\rangle =A\cup\left\{ 0\right\} =d\left(A'\cup\left\{ 0\right\} \right)$,
then $A'\cup\left\{ 0\right\} \cong\mathbb{Z}_{n/d}$. 
\end{proof}
An immediate consequence of the Corollary \ref{Cor.2} is that we
could consider only fully weight instead of a proper subgroup.

\begin{lemma}\label{L1}{[}Theorem 5.2 \cite{marc}{]} Let $G=H\oplus C_{n}^{r}$, where $H=C_{n_{1}}\oplus\cdots\oplus C_{n_{t}}$ with
$1<n_{1}|n_{2}|\cdots|n_{t}|n=exp(G)$ and $n_{t}<n$. Then, $D_{A}(G)=r+1$. 
\end{lemma}

A subsequence $T$ of $S$ is called a {\it maximal $A$-weighted zero-sum free subsequence} if 
$T$ is a subsequence of maximal length such that $T$ is $A$-weighted zero-sum free.

Below, we present an important result for the fully weighted Davenport
constant.

\begin{theorem}\label{T1}
\label{prop:cpr} All finite abelian group $G$ with exponent $n$
is $0$-complete with respect to $A=\left[1,n-1\right]$. 
\end{theorem}
\begin{proof}
According to Lemma \ref{L1}, we can write $D_{A}(G)=r+1$.
If $|S|\leq r$, then $N_{A,0}(S)\geq1\geq 2^{|S|-r}$.
If $|S|=r+1$, then there is an $A$-weighted zero-sum nonempty subsequence $T$ of $S$. Thus, $N_{A,0}(S)\geq 2=2^{|S|-r}$. Notice that if there is no $T|S$, such that $T$ is no a maximal $A$-weighted zero-sum free with $|T|=r$, then $N_{A,0}(S)>2^{|S|-r}$.

Suppose now $r+1<|S|$.

 We divide the proof in three cases:

\textbf{Case 1:} Let $S\in\mathcal{F}(G)$ be a sequence such that $o(g)<n$
for all $g|S$:

In this case, we have that each $g|S$ is an $A$-weighted zero-sum subsequence
and so $N_{A,0}(S)=2^{|S|}>2^{|S|-r}$.

\textbf{Case 2:} Let $S=TW\in\mathcal{F}(G)$ be a sequence such that
the elements of $T$ have order $n$ and $T$ is maximal $A$-weighted zero-sum free with $|T|<r$.

Then, for each element $g|W$, we have two possibilities:

\textbf{a)} If $o(g)<n$ and for $a_{g}=o(g)\in A$, then $g$ is an $A$-weighted zero-sum
subsequence.

\textbf{b)} If $o(g)=n$, then $Tg$ has an $A$-weighted zero-sum subsequence with $g$ being one of its elements.

In both possibilities, there is $V|Tg$, such that $V$ is an $A$-weighted zero-sum subsequence
whose coefficient of $g$ is $a_{g}\in A$. Then, $a_{g}g$ is an $A$-weighted sum
of some subsequence of $T$:$$a_{g}g={\sum}_{i\in I_g}a_i g_i; I_g\subset I_T.$$ Thus, for every $U|W$ nonempty, we have
$${\sum}_{g|U}a_{g}g={\sum}_{g|U} {\sum}_{i\in I_g}a_i g_i={\sum}_{g_i\in I_V}b_i g_i; I_V\subset I_T, \,\mbox{with}\, a_i,b_i\in A,$$
i.e., the $A$-weighted sum ${\sum}_{g|U}a_{g}g$ is an $A$-weighted sum of some subsequence $V_{U}$ of $T$. Therefore, $UV_{U}$ is an $A$-weighted zero-sum subsequence of $S$. Notice that if $V_{U}=\lambda$, then $U$ is an $A$-weighted zero-sum subsequence. Therefore, if we include the empty subsequence, we obtain
a minimum of $2^{|W|}=2^{|S|-|T|}$ distinct $A$-weighted zero-sum subsequences
of $S$. This proves that $N_{A,0}(S)>2^{|S|-r}$. 

\textbf{Case 3:} Let $S=TW\in\mathcal{F}(G)$ be a sequence such that
 $T$ is a maximal $A$-weighted zero-sum free and $|T|=r$.

As in Case $2,$ we obtain a minimum of $2^{|W|}=2^{|S|-|T|}$ distinct $A$-weighted zero-sum subsequences of $S$. Therefore, $N_{A,0}(S)\geq2^{|S|-r}$.
\end{proof}

\section{The structures of extremal sequences on the fully weighted\protect \linebreak{}
Davenport constant}

\global\long\def\labelenumi{(\roman{enumi})}

In this section, we consider that $G$ is of the form $H\oplus C_{n}^{r}$ with $\exp(H)<n$ and we will study sequence $S$, such that $N_{A,0}\left(S\right)=2^{\left|S\right|-D_{A}\left(G\right)+1}$ where $A=\left[1,n-1\right]$. The case $A=\{1\}$, for the general group of odd exponent, was studied by Chang {\it et al.} \cite{Chang}. 

Because $N_{A,0}\left(S\right)=2N_{A,0}\left(S0^{-1}\right)$ and if $o(g)<n$, then $N_{A,0}\left(S\right)=2N_{A,0}\left(Sg^{-1}\right)$,
it suffices to consider sequences $S$, such that $0\nmid S$ and $o(g)=n$ for all $g|S$. 

\begin{proposition}
\label{formadasequencia}
Let $G$ be a finite abelian group with $exp(G)=n$. If $S\in EC_{A,g}\left(\mathcal{F}\left(G\right)\right)$, with $A=[1,n-1]$, $0\nmid S$ and $o(g)=n$ for all $g|S$, then $r\leq\left|S\right|$ and there is $T=\prod_{i=1}^{r}g_i$ a maximal $A$-weighted zero-sum free, such that 
\begin{equation}
S=\prod_{i=1}^{r}g_{i}\prod_{j=1}^{k}h_{j},\label{eq:1-1}
\end{equation}
where $k\in\mathbb{N}_0$ and $b_{j}h_{j}=\sum_{i\in I_{j}}a_{i}g_{i}$
with $a_{i},b_{j}\in A$, $I_{j}\subset\left[1,r\right]$.
\end{proposition}

\begin{proof}
Let $S$ be a sequence over $G$, with $0\nmid S$, $o(g)=n$ for all $g|S$ and $N_{A,0}\left(S\right)=2^{\left|S\right|-D_{A}\left(G\right)+1}=2^{|S|-r}.$ If $|S|<r$, then as $\lambda$ is an $A$-weighted zero-sum subsequence of $S$ and $N_{A,0}\left(S\right)\geq 1>2^{|S|-r}$. Therefore, we can assume $|S|\geq r$. By Theorem \ref{T1}, Case 3, there is $T|S$, such that $T=\prod_{i=1}^{r}g_{i}$ is a maximal $A$-weighted zero-sum free, otherwise $N_{A,0}\left(S\right)>2^{|S|-r}$. 
According to Lemma \ref{L1}, we observe that if $g\nmid T$ then $a_lg=\sum_{i\in I\subset [1,r]}a_ig_i$, with $a_{i},a_l\in A$. If there is more than a set $I$ satisfying this, then by Theorem \ref{T1}, Case 3, it follows that $N_{A,0}\left(S\right)>2^{|S|-r}$. Therefore, there is only one set $I$ satisfying $a_lg=\sum_{i\in I\subset [1,r]}a_ig_i$. Thus, we conclude that $S=\prod_{i=1}^{r}g_{i}\prod_{j=1}^{k}h_{j}$, where $b_{j}h_{j}=\sum_{i\in I_{j}}a_{i}g_{i}$ with $a_{i},b_{j}\in A$, $I_{j}\subset\left[1,r\right]$. 
\end{proof}

Now, we prove the main result of this section.

\begin{theorem}
\label{thm11}Let $G$ be a finite abelian group with $exp(G)=n$ an odd number. If $S\in EC_{A,g}\left(\mathcal{F}\left(G\right)\right)$, with $A=[1,n-1]$, $0\nmid S$ and $o(g)=n$ for all $g|S$, then $r\leq\left|S\right|\leq2r$ and there is $T=\prod_{i=1}^{r}g_i$ a maximal $A$-weighted zero-sum free, such that 
\begin{equation}
S=\prod_{i=1}^{r}g_{i}\prod_{j=1}^{k}h_{j},\label{eq:1-1}
\end{equation}
 where $k\in\left[1,r\right]$, $b_{j}h_{j}=\sum_{i\in I_{j}}a_{i}g_{i}$
with $a_{i},b_{j}\in A$, $I_{j}\subset\left[1,r\right]$ and $I_{j}'s$
are pairwise disjoint ($I_{j}=\emptyset$ for all $j$ implies that
$S=\prod_{i=1}^{r}g_{i}$ ).
\end{theorem}

\begin{proof}
Let $S$ be a sequence over $G$ with $0\nmid S$, $o(g)=n$ for all $g|S$ and $N_{A,0}\left(S\right)=2^{\left|S\right|-D_{A}\left(G\right)+1}=2^{|S|-r}.$ We know, by Proposition \ref{formadasequencia}, that $S=\prod_{i=1}^{r}g_{i}\prod_{j=1}^{k}h_{j}$ where $k\in\mathbb{N}_0$, $b_{j}h_{j}=\sum_{i\in I_{j}}a_{i}g_{i}$ with $a_{i},b_{j}\in A$, $I_{j}\subset\left[1,r\right]$ and $T=\prod_{i=1}^{r}g_i$ is a maximal $A$-weighted zero-sum free. 

Now, we will prove that the $I_{j}'s$ are pairwise disjoint. If $\left|S\right|=D_{A}\left(G\right)-1=r$, then $N_{A,0}(S)=1$, $I_{j}=\emptyset$ for all $j\in [1,k]$ and $S=\prod_{i=1}^{r}g_{i}$. Suppose that $\left|S\right|=D_{A}\left(G\right)=r+1$ then, $I_{j}\neq\emptyset$ for only one $j$, $N_{A,0}(S)=2$ and
$S=\prod_{i=1}^{r}g_{i}h_{j}$. Finally, suppose $S=\prod_{i=1}^{r}g_{i}\prod_{j=1}^{k}h_{j}$
with $k\geq2$ and $I_{j_1}\cap I_{j_2}\neq\emptyset$ for some $j_1,j_2\in\left[1,k\right]$, with $j_1\neq j_2$ and
where $$a_{j_1}h_{j_1}=\sum_{i\in I_{j_1}}a_{i}g_{i} \mbox{ and } a_{j_2}h_{j_2}=\sum_{i\in I_{j_2}}b_{i}g_{i}$$ with $a_{j_1},a_{j_2},a_i,b_{i}\in A$.

Using the same argument of Theorem \ref{T1}, Case 3, we have $\tbinom{k}{0}+\tbinom{k}{1}+\cdots+\tbinom{k}{k}=2^{k}=2^{\left|S\right|-r}$
$A$-weighted zero-sum subsequences of $S$. Since $I_{j_1}\cap I_{j_2}\neq\emptyset$, we have $I_{x}, I_{y}\subset I_{j_1}\cup I_{j_2}$ such that 
\begin{equation}\label{eq.7}
a_{j_1}h_{j_1}+a_{j_2}h_{j_2}=\sum_{i\in I_{x}}c_{i}g_{i}\,\,\mbox{and}\,\,a_{j_1}h_{j_1}-a_{j_2}h_{j_2}=\sum_{i\in I_{y}}d_{i}g_{i}, \,c_i,d_i\in A.  
\end{equation}
 
Notice that $\exp(G)=n$ is an odd number, it follows that $I_{x}\neq\emptyset$ or $I_y\neq\emptyset$.

If $I_x\neq I_y$, then there is a new $A$-weighted zero-sum subsequence of $S$ and therefore $N_{A,0}\left(S\right) >2^{|S|-r}$, which is a contradiction.
Now, suppose that $I_{x}=I_y$ and take $g_l|\prod_{i\in I_{j_1}\cap I_{j_2} }g_{i}$ (observe that $c_l\neq0, d_l\neq 0$ in \eqref{eq.7}). Consider $Tg_l^{-1}=\prod_{i=1}^{r+1}g_ig_l^{-1}$, where $g_{r+1}=h_{j_2}$. If $T$ is not a maximal $A$-weighted zero-sum free, then there is $\bar{I}_{j_2}\subset \left[1,\dots,r+1\right]\backslash \left\{l\right\}$ such that $z_{j_2}h_{j_2}=\sum_{i\in \bar{I}_{j_2}}s_{i}g_{i}$, i.e., we can obtain a new $A$-weighted zero-sum subsequence of $S$ and thus $N_{A,0}\left(S\right) >2^{|S|-r}$, which is a contradiction. If $T$ is a maximal $A$-weighted zero-sum free, then by Lemma \ref{L1} we have $\bar{I}_{j_1}\subset \left[1,\dots,r+1\right]\backslash \left\{l\right\}$ such that $v_{j_1}h_{j_1}=\sum_{i\in \bar{I}_{j_1}}u_{i}g_{i}$, i.e., we can obtain a new $A$-weighted zero-sum subsequence of $S$. Therefore, we have $N_{A,0}\left(S\right) >2^{|S|-r}$ again, which is a contradiction.

We observe that if $k>r$, then there are $I_{j_1}$ and $I_{j_2}$
with $j_1\neq j_2$, such that $I_{j_1}\cap I_{j_2}\neq\emptyset$. Therefore, $N_{A,0}\left(S\right)>2^{|S|-r}$. Thus, $r\leq|S|\leq 2r$.
\end{proof}

Because of the aforementioned, we make the following
conjecture: 
\begin{conjecture}
All finite abelian group $G$ with exponent $n$ is $0$-complete
with respect to $A\subseteq\mathbb{Z}\backslash\left\{ kn:k\in\mathbb{Z}\right\} $. 
\end{conjecture}

\section*{Acknowledgement}
We would like to thank referees for all suggestions.


\begin{thebibliography}{99}

\bibitem{Adh1}
S. D. Adhikari, Y.G Chen, J. B. Friedlander, S. V. Konyagin, F. Pappalardi, Contributions to zero-sum problems, \textit{Discrete Math.} {\bf306} (2006), no. 1, 1--10.

\bibitem{Bal}
E. Balandraud, An addition theorem and maximal zero-sum free set in Z/pZ, \textit{Israel J. Math.} {\bf188} (2012), 405--429.
\bibitem{Bia}
A. Bialostocki and M. Lotspeich, Some developments of the Erd\"{o}s-Ginzburg-Ziv Theorem I, Sets, Graphs and Numbers, \textit{Coll. Math. Soc. J. Bolyai} {\bf60} (1992), 97--117. 

\bibitem{Cao}
H.Q. Cao and Z.W. Sun, On the number of zero-sum subsequences, \textit{Discrete Math.} {\bf307} (2007), 1687--1691.

\bibitem{Chang} G. J. Chang, S. Chen, Y. Qu, G. Wang, H. Zhang, On the number of subsequence with a give sum in finite abelian group, \textit{Electron. J. Combin.} \textbf{18} (2011), no. 1, paper 133.

\bibitem{Dav} H. Davenport, On the addition of residue classes, \textit{J. Lond. Math. Soc.} {\bf10} (1935),
30--32.

\bibitem{EGZ} P. Erd\"{o}s , A. Ginzburg, A. Ziv, Theorem in the additive number theory, \textit{Bulletim Research Council Israel} {\bf 10F}, 41--43, 1961.

\bibitem{Fur}
Z. F\"uredi and D.J. Kleitman, The minimal number of zero sums, Combinatorics, Paul Erd\"{o}s is Eighty, \textit{J. Bolyai Math. Soc.} (1993), 159--172.


\bibitem{Gao2} W.D. Gao,  On the number of zero-sum subsequences, \textit{Discrete Math.} \textbf{163} (1997), 267--273.

\bibitem{Gao3}
W.D. Gao, On the number of subsequences with given sum, \textit{Discrete Math.} {\bf195} (1999), 127--138.

\bibitem{Gao5} W.D. Gao and A. Geroldinger, On the number of subsequences with given sum of
sequences over finite abelian p-groups, \textit{Rocky Mountain J. Math.} \textbf{37} (2007), 1541--
1550.

\bibitem{Gao6} W.D. Gao and J.T. Peng, On the number of zero-sum subsequences of restricted size, \textit{Integers} \textbf{9} (2009), 537--554.

\bibitem{Ger1} A. Geroldinger and F. Halter-Koch, Non-unique factorizations, Combinatorial and
Analytic Theory, \textit{Pure and Applied Mathematics} {\bf278}, Chapman \& Hall/CRC,
2006.

\bibitem{Ger2}
A. Geroldinger, \textit{Additive group theory and non-unique factorizations, Combinatorial
Number Theory and Additive Group Theory}, Advanced Courses in Mathematics, CRM Barcelona, Birkhauser, (2009), 1--86.

\bibitem{Gry1}
D.J. Grynkiewicz, On the number of m-term zero-sum subsequences, \textit{Acta Arith.} {\bf121} (2006), 275--298.

\bibitem{Gry2}
D.J. Grynkiewicz, E. Marchan and O. Ordaz, Representation of finite abelian group
elements by subsequence sums, \textit{J. Theor. Nombres Bordeaux} {\bf21} (2009), 559--587.

\bibitem{Gui}
D.R. Guichard, Two theorems on the addition residue classes, \textit{Discrete Math.} {\bf81}
(1990), 11--18.

\bibitem{kis}
M. Kisin, The number of zero sums modulo m in a sequence of length n, \textit{Mathematica}
{\bf41} (1994), 149--163.

\bibitem{marc}
L. E. Marchan, O. Ordaz and W. A. Schmid, Remarks on the plus--minus weighted Davenport constant, \textit{International Journal of Number Theory} {\bf10} (05) (2014), 1219--1239.


\bibitem{OlsonI} J. E. Olson. A combinatorial problem on finite abelian groups I, \textit{Jornal of Number Theory} \textbf{1} (1969), 8--10.

\bibitem{OlsonII} J. E. Olson. A combinatorial problem on finite abelian groups II, \textit{Jornal of Number Theory} \textbf{1} (1969), 195--199.

\end{thebibliography}
\end{document}